\documentclass[12pt,reqno]{amsart}
\usepackage{amssymb}
\usepackage[all]{xy}

\newtheorem{thm}{Theorem}[section]
\newtheorem{corol}[thm]{Corollary}
\newtheorem{lemma}[thm]{Lemma}
\newtheorem{prop}[thm]{Proposition}
\newtheorem{defin}[thm]{Definition}
\newtheorem{conj}[thm]{Conjecture}

\newcommand{\Oc}{\mathcal O}

\def\ker{\operatorname{ker}}
\def\rk{\operatorname{rk}}

\def\coker{\operatorname{coker}}

\def\dim{\operatorname{dim}}
\def\id{\operatorname{id}}

\newcommand\grass{\mbox{Gr}}
\newcommand\hgrass{{\mathfrak{Gr}}}
\newcommand{\cO}{{\mathcal O}}

\newcommand{\fE}{{\mathfrak E}}
\newcommand{\fF}{{\mathfrak F}}

\newcommand{\fG}{{\mathfrak G}}

\newcommand{\fQ}{{\mathfrak Q}}

\begin{document}

\title[Higgs bundles and fundamental group schemes]{Higgs bundles and fundamental 
group schemes}

\author[I. Biswas]{Indranil Biswas}

\address{School of Mathematics, Tata Institute of Fundamental
Research, Homi Bhabha Road, Bombay 400005, India}
\email{indranil@math.tifr.res.in}

\author[U. Bruzzo]{Ugo Bruzzo}

\address{Scuola Internazionale Superiore di Studi Avanzati (SISSA), Via Bonomea 265, 34136 Trieste, Italy;
Istituto Nazionale di Fisica Nucleare, Sezione di Trieste}

\email{bruzzo@sissa.it}

\author[S. Gurjar]{Sudarshan Gurjar}

\address{Department of Mathematics, Indian Institute of Technology, Mumbai  400076, India
}
\email{sgurjar@math.tifr.res.in}

\date{}



\begin{abstract}  Relying on a notion of ``numerical effectiveness'' for Higgs bundles, we show that the category of ``numerically flat'' Higgs vector bundles on a smooth projective variety $X$ is a Tannakian category. We introduce the associated group scheme, that we call the ``Higgs fundamental group scheme of $X$,'' and show that its properties are related to a conjecture about the vanishing of the Chern classes of numerically flat Higgs vector bundles.
\end{abstract}

\maketitle

\section{Introduction}

Given a projective scheme $X$ over a field $k$, a line bundle $L$ on $X$ is said to be numerically effective
(abbreviated as ``nef'')
if $\deg f^\ast L \,\ge\, 0$ for every morphism $f\,\colon \,C\,\longrightarrow\, X$, where $C$ is an
irreducible smooth projective curve.
 A notion of numerical effectiveness for a vector bundle $E$ can be given by asking that the
relative hyperplane bundle $\mathcal{O}_{\mathbb{P}(E)}(1)$ on the projective bundle $\mathbb P(E)$ is nef  \cite{hartshorne-ample,lazarsfeld-I,lazarsfeld-II}.
More generally, if $\rk E = r$, one can consider for every $k$, with $0<k<r$, the Grassmann bundle
$\operatorname{Gr}_k(E)\,\longrightarrow\, X$ that parameterizes the quotients of fibers of $E$ of
dimension $k$.
The universal quotient bundle $Q_{k,E}$ of rank $k$ on $\operatorname{Gr}_k(E)$ 
   satisfies the well-known property that for any morphism $g \,\colon\, Y \,\longrightarrow\, X$
if $F$ is a rank $k$ quotient of $g^\ast E$, then there is a morphism $h\,\colon\, Y
\,\longrightarrow\, \operatorname{Gr}_k(E)$ which covers $g$ and satisfies the condition that
$F\,\simeq\, h^\ast Q_{k,E}$. It turns out that $E$ is nef if and only if all universal quotients $Q_{k,E}$ are nef.

One can consider vector bundles $E$ such that both $E$ and its dual $E^{^*}$ are nef. These are
called {\em numerically flat} bundles. The numerically flat bundles enjoy very special
properties; they have vanishing rational Chern classes \cite{DPS94}, and they form a Tannakian
category $\mathbf{NF}(X)$. The associated group scheme $G$ defined by the
property that $\mathbf{NF}(X)$ is
the category of representations of $G$ was introduced in \cite{BH,langer-fund-I}. 

Building on ideas already contained in \cite{bruzzo-hernandez-dga}, in \cite{bruzzo-grana-cmm} a 
definition of ``Higgs numerical effectiveness'' (``H-nef'' for short) was given (however
the basics of this theory in their final form were presented subsequently
in \cite{bruzzo-grana-adv}). Given a Higgs vector 
bundle $\mathfrak E \,=\, (E,\, \phi)$, and any $0<k<r$, the idea is to use the Higgs field $\phi$ to
construct a closed 
subscheme $\hgrass_k(\mathfrak E) \,\subset\,\operatorname{Gr}_k(E)$, with the property that 
a rank $k$ quotient $F$ of $E$ is a Higgs quotient of $\fE$ (i.e., the kernel corresponding to it is 
$\phi$--invariant) if and only if the image of the associated section of $\operatorname{Gr}_k(E)$
is contained
in $\hgrass_k(\mathfrak E)$. The universal quotient bundle $Q_{k,E}$ restricts to 
$\hgrass_k(\mathfrak E) $ to yield a {\em universal Higgs quotient bundle} $\fQ_{k,\fE}$. This opens the way to define Higgs-numerically effective Higgs bundles, in 
terms of a recursive positivity property of the bundles $\fQ_{k,\fE}$ (see Definition 
\ref{defHnef} for a precise statement). Higgs-numerically flat bundles (H-nflat Higgs bundles) are then 
defined as H-nef Higgs bundles for which the dual Higgs bundle is H-nef as well. It turns out
that the H-nflat Higgs bundles on a smooth projective variety $X$ make up a Tannakian category
$\mathbf{HNF}(X)$. We  denote by $\pi^H_1(X,x)$
the associated group scheme, where  $x\,\in\, X$ is the base point needed to define
the fiber functor, and call it the {\em Higgs fundamental group scheme of } $X$.

In Section \ref{last} we study some basic properties of this group. It turns out that this group 
is related to a conjectured property of Higgs bundles 
\cite{bruzzo-hernandez-dga,bruzzo-grana-adv}. For vector bundles $E$ on a projective manifold $X$, 
the following property is known to be true 
\cite{nakayama,bruzzo-hernandez-dga,biswas-bruzzo-imrn}.

\begin{thm}\label{thmvb}
The following conditions are equivalent:
\begin{itemize} \item for every morphism $f\,\colon\, C\,\longrightarrow\, X$, where $C$ is a
smooth irreducible projective curve, the bundle $f^\ast E$ is semistable;
\item $E$ is semistable with respect to some polarization, and the characteristic class
$$\Delta(E) = c_2(E) -\frac{r-1}{2r}c_1(E)^2 \in H^4(X,\mathbb Q)$$
vanishes (here $r = \rk E$).
\end{itemize}
\end{thm}

For Higgs bundles, it is known that the second condition implies the first 
\cite{bruzzo-hernandez-dga,bruzzo-grana-adv}, but the fact that the first implies the second is 
an open conjecture (see \cite{bruzzo-logiudice} for the characterization of a class of varieties 
for which this conjecture holds). It is equivalent to the fact that H-nflat Higgs bundles have 
vanishing rational Chern classes (see Corollary \ref{equiv}). For future convenience, we explicitly state this conjecture.

\begin{conj} Let $\mathfrak E = (E,\phi)$ be a Higgs bundle on a smooth projective variety $X$,
such that for every morphism $f\,\colon\, C\,\longrightarrow\, X$, where $C$ is a
smooth irreducible projective curve, the Higgs bundle $f^\ast \fE$ is semistable. Then
$\Delta(E) =0$.\label{conj}
\end{conj}

As we discuss in Section \ref{last},  the above conjecture is 
also related to the following product formula for the Higgs fundamental group scheme: if $X$, 
$Y$ are smooth projective varieties over a field $k$, and $x$, $y$ are points in $X$, $Y$, 
respectively, then $$\pi_1^H(X\times_k Y,\,(x,y)) \,\simeq\, 
\pi_1^H(X,\,x)\times\pi_1^H(Y,\,y)\,.$$

\smallskip
{\bf Acknowledgements.} We thank Beatriz Gra\~na Otero for useful discussions.
 I.B. is supported by a J.C.~Bose Fellowship. U.B.'s research is partly supported by
INdAM-GNSAGA. S.G.~would like to thank the International Centre for
Theoretical Physics, Trieste for a year long postdoctoral position during
which time the project was initiated. The project continued when the
author visited the Centre for Quantum Geometry of Moduli Spaces, Aarhus.
He would like to thank the Center for Quantum Geometry of Moduli Spaces
for their hospitality, supported by a Center of Excellence grant from the
Danish National Research Foundation (DNRF95) and by a Marie Curie
International Research Staff Exchange Scheme Fellowship within the 7th
European Union Framework Programme (FP7/2007-2013) under grant agreement n
612534, project MODULI - Indo European Collaboration on Moduli Spaces.

U.B.~is a member of VBAC.

\bigskip
\section{Numerically effective Higgs bundles}

\textbf{Notation.} \ 
Unless otherwise stated, $X$ will denote a smooth projective variety
of dimension $n$ defined over an algebraically closed field $k$ of
characteristic zero. The cotangent bundle of $X$ will be denoted by $\Omega^1_X$. We shall usually denote
by a Gothic letter, such as $\fE$, a pair $(E,\,\phi)$, 
where $E$ is a coherent sheaf and $\phi$ is a Higgs field on $E$ (see Definition \ref{defHnef}).
So a roman letter will denote the underlying coherent sheaf  of a Higgs sheaf. 

We fix a very ample line bundle on $X$ and denote by $H$ its
numerical class. The degree of a torsion-free coherent $\cO_X$--module $F$ is defined as
to be
$$\deg F \,:=\, c_1(F)\cdot H^{n-1}\, ,$$
and if $\rk F\,\not=\, 0$, one defines the {\it slope} of $F$ to be
$$\mu( F)\, := \,\frac{\deg F}{\rk  F}\, .$$

\begin{defin} A Higgs sheaf $\fE$ on $X$ is a pair $(E,\phi)$, where $E$ is a torsion-free
coherent sheaf on $X$  and $$\phi \,\colon\, E \,\longrightarrow\, E \otimes \Omega_X^1$$ is a homomorphism of
$\cO_X$-modules
such that $\phi\wedge\phi\,=\,0$. A Higgs subsheaf
 of a Higgs sheaf $\fE=(E,\phi)$ is a pair $(G,\phi')$, where $G$ is a subsheaf of $E$
such that $\phi(G)\,\subset\, G\otimes\Omega_X^1$, and $\phi'\,=\, \phi\vert_G$.
A Higgs bundle is a Higgs sheaf $\fE $ such that $E$ is a
locally-free $\Oc_X$-module. If $\fE\,=\,(E,\phi)$ and $\fG\,=\,(G,\psi)$ are
Higgs sheaves, a morphism $f\,\colon\,\fE\,\longrightarrow\, \fG$ is a homomorphism of $\Oc_X$-modules
$f\,\colon\, E\,\longrightarrow\, G$ such that the diagram
$$\xymatrix{
E\ar[r]^f \ar[d]_\phi & G \ar[d]^\psi \\
E\otimes\Omega^1_X \ar[r]^{f\otimes\id} & G\otimes\Omega^1_X }
$$
commutes.
\label{defHnef}
\end{defin}

\begin{defin} A Higgs sheaf $\fE=(E,\phi) $ is semistable (respectively, stable) if
$\mu(G)\le \mu(E)$ (respectively, $\mu(G)< \mu(E)$) for every
Higgs subsheaf $(G,\phi')$ of $\fE$ with $0\, <\,\rk  G\, <\, \rk  E$.\end{defin}

From now on, unless otherwise stated, by semistability of a Higgs bundle we will mean semistability in the Higgs sense (as in the above definition).
Let us recall the definition of numerical effective vector bundles
on a projective variety $X$. A line bundle $L$ on $X$ is said to
be numerically effective (nef for short) if, for every pair $(C\, ,f)$, where
$C$ is a smooth projective irreducible curve and $f\,\colon\, C \,\longrightarrow\, X$ is a morphism,
the line bundle $f^\ast L$ on $C$ has nonnegative degree. A vector bundle $E$ is
numerically effective
if the hyperplane line bundle $\cO_{\mathbb P(E)}(1)$ on the projectivization
$\mathbb P(E)$ of $E$ is numerically
effective. For the main properties of numerically effective vector bundles
see e.g.\ \cite{hartshorne-ample,lazarsfeld-I,lazarsfeld-II}.

Let $E$ be a vector bundle of rank $r$ on $X$, and let $s\, <\, r$ be a positive
integer. We shall denote by $\grass_s(E)$ the Grassmann bundle on $X$ parameterizing
quotients of fibers of $E$ of dimension $s$. Let $p_s \,:\, \grass_s(E) \,\longrightarrow\, X$ be the
natural projection. There is a universal short exact sequence
\begin{equation}\label{univ}
0 \,\longrightarrow\, S_{r-s,E} \,\stackrel{\psi}{\longrightarrow}\, p_s^* E \,
\stackrel{\eta}{\longrightarrow}\, Q_{s,E}
\,\longrightarrow\, 0
\end{equation}
of vector bundles on $\grass_s(E)$, with $S_{r-s,E}$ being a universal   
subbundle of rank $r-s$ and $Q_{s,E}$ a universal quotient of rank $s$. 
Given a Higgs bundle $\fE\,=\,(E,\phi) $, we have the closed subschemes
$\hgrass_s(\fE)\,\subset\, \grass_s(E)$ parameterizing rank
$s$ locally-free Higgs quotients, i.e., locally-free quotients of
$E$ whose corresponding kernels are $\phi$-invariant. In other words,
$\hgrass_s(\fE)$ \emph{(the Grassmannian of locally free rank $s$ Higgs
quotients of $\fE$)} is the closed subscheme of $\grass_s(E)$ defined by
the vanishing of the composed morphism
\begin{equation}\label{lambda}
(\eta\otimes\text{Id})\circ p_s^\ast(\phi) \circ \psi\,\colon\, S_{r-s,E}\,\longrightarrow\, Q_{s,E}\otimes
 p_s^\ast\Omega_X^1\, .
\end{equation}
Let $\rho_s\, :=\, p_s\vert_{\hgrass_s(\fE)}\, :\, \hgrass_s(\fE)\,\longrightarrow\, X$ be the
restriction. The restriction of \eqref{univ} to $\hgrass_s(\fE)$ provides the universal
exact sequence
\begin{equation}\label{g1}
0\,\longrightarrow\, \mathfrak S_{r-s,\fE}\,\stackrel{\psi}{\longrightarrow} \,\rho_s^\ast \fE\,
\stackrel{\eta}{\longrightarrow} \mathfrak Q_{s,\fE}\,\longrightarrow\, 0\, ,
\end{equation}
with $\mathfrak Q_{s,\fE}\,:=\, Q_{s}\vert_{\hgrass_s(\fE)}$ 
being equipped with the quotient Higgs field induced by the Higgs field
$\rho_s^\ast \phi$.
The universal property satisfied by
$\hgrass_s(\fE)$ is that given any morphism of $k$-varieties $f: T \rightarrow X$, $f$ factors through
$\hgrass_s(\fE)$ if and only if the pullback $f^*(E)$ admits a Higgs quotient of rank $s$. In that case the
pullback of the above universal sequence on $\hgrass_s(E)$ gives the desired quotient of $f^*(E)$.

\begin{defin}\label{moddef}  
A Higgs bundle $\fE$ of rank one is said to be Higgs-numerically
effective  (H-nef for short) if it is numerically effective in the usual sense. If
$\rk \fE \geq 2$, we inductively define H-nefness by requiring that
\begin{enumerate}
\item all Higgs bundles $\fQ_{s,\fE}$ are Higgs-nef (see \eqref{g1}) for all $s$, and

\item the determinant line bundle $\det(E)$ is nef.
\end{enumerate}
If both $\fE$ and $\fE^\ast$ are Higgs-numerically effective, $\fE$ is said to
be Higgs-numerically flat (H-nflat).
\end{defin}

Definition \ref{moddef} immediately implies that the first Chern class of an
H-numerically flat Higgs bundle is numerically equivalent to zero. Note that if
$\fE\,=\,(E,\,\phi)$, with $E$ nef in the usual sense, then $\fE$ is H-nef. Moreover, if
$\phi\,=\,0$, the Higgs bundle $\fE\,=\,(E,\,0)$ is H-nef if
and only if $E$ is nef in the usual sense (as in this case the Higgs Grassmannian coincides
with the usual Grassmannian bundle, and the respective universal bundles coincide).

We recall that in the case of ordinary vector bundles, nefness is defined using only the
hyperplane bundle.
Let us motivate why one should consider the behavior of the universal Higgs quotients 
of all ranks, and therefore introduce Higgs Grassmannians corresponding to quotients of 
all ranks. In the case of ordinary bundles, if the hyperplane bundle is nef, then the 
universal quotients of all ranks are nef as well; indeed, if a vector bundle $E$ is 
nef, its pullback to the Grassmannian $\grass_s(E)$ is nef, and the quotient $Q_{s,E}$ 
(see equation \eqref{univ}) is nef too. This is not the case for Higgs bundles, as 
the following example shows.

Let $\fE\,=\,(E,\,\phi)$  be a rank three nilpotent Higgs bundle on a smooth  
projective curve $C$, having  the form  $E\,=\,L_1\oplus L_2\oplus L_3$, where 
each $L_i$ is a line bundle, and  $\phi(L_1)\,\subset\, L_2\otimes\Omega^1_C$, $\phi(L_2)\,\subset\,  
L_3\otimes\Omega^1_C$, $\phi(L_3)\,=\,0$. Denote by $d_i$ the degree of $L_i$, and assume that
$d_1+d_2+d_3\,=\,0$.
The computations in Section 3.4 of  \cite{bruzzo-hernandez-dga} show that the hyperplane bundle of $\fE$, restricted to the
Higgs Grassmannian
$\hgrass_1(\fE)$, is nef if $2d_1-d_2-d_3\ge 0$, while the rank two universal quotient
on $\hgrass_2(\fE)$ is
nef if and only if 
$d_1+d_2-2d_3\ge 0$. There exist values of the degrees for which the first inequality holds and the second
does not. For instance, if $C$ has genus 3, one can take $d_1=d_3=1$, $d_2=-1$. Note that by Riemann-Roch theorem
$h^0(C,K_C)\,>\, 0$ and hence an effective divisor exists in the linear system $|K_C|$ (of degree 4). To ensure that there exists 
a nonzero Higgs morphism, write $K_C=(x_1+x_2+x_3+x_4)$, with $x_i$ points in $C$, 
and take $L_1\,=\,(x_1)$, $L_2\,=\,-(x_2+x_3)$, $L_3\,=\,(x_4)$. 

Moreover, one includes the  condition that $\det(E)$ is  nef in Definition \ref{moddef}  
to prevent the existence of H-nef Higgs bundles of negative degree. One such example is
provided by a Higgs bundle $\fE=(E,\phi)$ on a smooth projective curve, with $E=L_1\oplus L_2$ (where $L_1$,
$L_2$ are line bundles),
and $\phi\,\colon\, L_1\,\longrightarrow\, L_2\otimes \Omega_X^1$, $\phi(L_2)\,=\,0$. 
As shown in \cite{bruzzo-hernandez-dga}, $\fE$
has only  two Higgs quotients, i.e., $L_1$ and 
$$ Q = \coker (\phi\otimes\mbox{id}) \,\colon\, E\otimes T_X \,\longrightarrow\, E$$
modulo torsion; the latter one will be denoted by $\overline Q$.
Note that $\deg({\overline Q}) \ge \deg(L_1)$. If the genus of $X$ is at least 2, one
can for instance take $\deg(L_1)\,=\,0$ and $\deg(L_2) \,=\,-2$.
Then $\fE$ satisfies all the conditions in the definition of H-nefness
except the one which says that $\det(E)$ is nef.

\bigskip
\section{Properties of H-nef Higgs bundles}

We give a few properties of H-nef Higgs bundles.
\begin{prop}\label{results}\mbox{}
\begin{enumerate}
\item[(i)] An H-numerically flat Higgs bundle is semistable. 

\item[(ii)]  Let  $\fE\,=\,(E,\phi)$ be a Higgs bundle whose first Chern class is numerically equivalent to zero. 
Assume that for all morphisms $f\,\colon\, C \,\longrightarrow\, X$, where $C$ is a smooth irreducible projective curve,
the pullback $f^\ast \fE$ is semistable. Then $\fE$ is H-nflat. 

\item[(iii)]   Let $$ 0 \,\longrightarrow\, \fF \,\longrightarrow\, \fE \,\longrightarrow\, \fG
\,\longrightarrow\, 0 $$
be a short exact sequence of Higgs bundles. If $\fF$ and $ \fG$ are H-nflat, so is $\fE$.

\item[(iv)]  If  $\fE$ and $\fG$ are H-nflat Higgs bundles, then the tensor product
$\fE\otimes \fG$ is H-nflat. 
\end{enumerate}
\end{prop}

\begin{proof}
(i)  and (ii) are    Proposition 8.8  and Lemma 8.7 in  \cite{bruzzo-grana-adv}, respectively.

(iii) Let  $C$ be a smooth irreducible projective curve, and $f\,\colon\, C \,\longrightarrow\, X$ a morphism. Then the sequence
$$ 0\,\longrightarrow\, f^\ast\fF \,\longrightarrow\, f^\ast \fE \,\longrightarrow\, f^\ast\fG
\,\longrightarrow\, 0 $$
is exact. As $f^\ast\fF$ and $f^\ast\fG $ are H-nflat,  their first Chern classes are numerically
equivalent to zero and they
are semistable by part (i). It follows that $f^\ast \fE$ is semistable as well. Hence $\fE$
is H-nflat.

(iv) Again, let  $C$ be a smooth irreducible projective curve, and $f\,\colon\, C \,\longrightarrow\, X$ a morphism. By the same
argument
as in part (iii), $f^\ast\fE$ and $ f^\ast\fG $ are  semistable. Then, as shown in \cite{balaji-para}, 
$f^\ast\fE\otimes f^\ast\fG \simeq f^\ast (\fE\otimes \fG )$ is semistable as well. 
 Moreover, 
$$c_1(E\otimes G) = \rk E\cdot c_1(G) + \rk G\cdot c_1(E) \equiv 0.$$
So by part (ii), $\fE\otimes \fG$ is H-nflat.
\end{proof}

A corollary to Proposition \ref{results} is that Conjecture  \ref{conj} is equivalent to the property that all H-nflat Higgs bundles have vanishing rational Chern classes (the analogous fact for vector bundles was proved in \cite{bruzzo-hernandez-dga}).

\begin{corol} \label{equiv}The following facts are equivalent.
\begin{enumerate} \item[(i)]  If $\fE=(E,\phi)$ is a Higgs bundle on $X$, and 
for all morphisms $f\,\colon\, C \,\longrightarrow\, X$, where $C$ is a smooth irreducible projective curve,
the pullback $f^\ast \fE$ is semistable, then $\Delta(E)=0$. 
\item[(ii)]  All rational Chern classes of an H-nflat Higgs bundle vanish.
\end{enumerate}
\end{corol}

\begin{proof} Assume that (i) holds, and let  $\fE=(E,\phi)$ be H-nflat. Then it follows easily that $f^*(\fE)$ is also H-nflat. 
By Proposition \ref{results} (i), every pullback $f^\ast \fE$ is semistable. Since (i) holds, we have $\Delta(E)=c_2(E)=0$. By Theorem 2 in \cite{simpson-local}, $\fE$ has a filtration whose quotients are flat, so that all Chern classes of $\fE$ vanish.

Conversely, assume that (ii) holds, and let $\fE$ be a Higgs bundle such that all pullbacks $f^\ast\fE$ are semistable. We may assume that 
$\fE$ has vanishing first Chern class by replacing it with its endomorphism bundle. By Proposition \ref{results}  (ii), $ \fE$ is H-nflat. Since (ii)
holds, we have in particular $\Delta(E)=0$.
\end{proof}

Actually Proposition \ref{results} (iv) can be generalized to H-nef Higgs bundles. We shall use the
Harder-Narasimhan filtration for Higgs bundles on curves \cite{arijit-partha-2005}. Given a Higgs bundle $\fE$ on a smooth,
projective curve $Y$ defined over $k$, 
 there exists a unique filtration of $\fE$ by Higgs subsheaves $0 \subset \fE_1  \subset \fE_2 \subset  \cdots
\subset \fE_r = \fE$   such that 
the successive quotients $\fE_i/\fE_{i-1}$ are semistable as Higgs sheaves with  their slopes satisfying the
inequalities $\mu(E_i/E_{i-1}) > \mu(E_{i+1}/E_i)$  for all $i$. Set $\mu_{max}(\fE)=\mu(E_1)$ and
$\mu_{min}(\fE)=\mu(E_r/E_{r-1})$.

In the rest of the paper, by the Harder-Narasimhan filtration of a Higgs bundle we will mean a filtration as
above. 

 The Harder-Narasimhan filtration for Higgs bundles on a curve has the following basic properties,
analogous to those of the Harder-Narasimhan filtration for torsion-free sheaves (see [1]): 

 1. If $\fE$ and $\fF$ are two Higgs bundles, $\mu_{\text{max}}(\fE \otimes \fF) = \mu_{\text{max}}(\fE) +
\mu_{\text{max}}(\fF)$ and $\mu_{\text{min}}(\fE \otimes \fF) = \mu_{\text{min}}(\fE) +
\mu_{\text{min}}(\fF)$.
 
2. If $\fE^{\bullet}= \{\fE_0\subset \fE_1\subset \fE_2\subset\cdots \subset \fE_s\}$ is any filtration of $\fE$
such that each filter is preserved by the Higgs field, the
Harder-Narasimhan polygon for $\fE^{\bullet}$ lies under the Harder-Narasimhan polygon for $\fE$. (See
[13] for the definition of the Harder-Narasimhan polygon.)
\smallskip

The following lemma generalizes a criterion holding for numerically effective bundles 
\cite{barton-1971}.

\begin{lemma}\label{le1}
Let $\fE$ be a Higgs bundle on $X$. Then  $\fE$ is H-nef if and only if for any morphism
$f\colon C \rightarrow X$, where $C$ is a smooth projective irreducible curve, one has
$\mu_{\text{min}}(f^{*}\fE) \,\geq\, 0$.
\end{lemma}
 
\begin{proof} 
Suppose $\fE$ is H-nef. Let $f\colon C \rightarrow X$ be any morphism from a smooth projective irreducible
curve $C$ to $X$. Let
$$0\subset \fE_1\subset \fE_2\subset \cdots \subset \fE_r= f^\ast\fE$$ be the 
Harder-Narasimhan filtration of the pullback of $\fE$ to $C$. Since $\fE$ is H-nef, it follows that $\deg
f^\ast\fE\geq0$. Let $s=\rk (\fE_r/\fE_{r-1})$. By the universal property of the Higgs Grassmannian
$\hgrass_s(f^\ast\fE)$, 
from    the natural quotient morphism $\phi_r\,\colon\, \fE_r\,\longrightarrow\, \fE_r/\fE_{r-1}$ we get a morphism
$C \,\longrightarrow\, \hgrass_s(f^\ast\fE)$ such that the pullback of the universal quotient  on $\hgrass_s(f^\ast\fE)$
coincides with $\phi_r$. By the H-nefness of $\fE$, it follows that 
$\deg ( \fE_r/\fE_{r-1}) \geq 0$.

Conversely, suppose $\fE$ has the property that for any smooth projective irreducible curve $C$ and any
morphism $f\,\colon\, C \,\longrightarrow\, X$, $\mu_{\text{min}}(f^{*}(\fE)) \,\geq\, 0$. 
We want to show that $\fE$ is an  H-nef Higgs bundle on $X$. The assumption on $\fE$ implies that the  degree of
$\fE$ is non-negative on  every curve, so that  $\det f^\ast\fE$ is nef and hence H-nef. 

To prove the other condition in the definition of H-nef bundles,
we recall that, as explained in \cite{bruzzo-grana-adv}, the H-nefness of $\fE$ is equivalent
to the nefness, in the usual sense, of a collection of line bundles $\mathfrak L_S$, each defined on a scheme
$S$ equipped with a projection $\pi_S\,\colon\, S \,\longrightarrow\, X$ (these line bundles are obtained by successively taking
the universal Higgs quotient until one reaches the rank one quotient bundles). Let 
$\psi_S\,\colon\, \pi_S^\ast\fE \,\longrightarrow\, \mathfrak L_S$ denote the quotient morphism, and let $g\colon C \rightarrow
S$ be any morphism, where $C$ is a smooth curve. The pullback of  $\psi_S$ to $C$ produces a  
quotient $f^\ast\fE \,\longrightarrow\, \fF$ on $C$, where $f=\pi_S\circ g$.   Let $\fF'$ denote the kernel of this quotient.
By property (1) of the Harder-Narasimhan filtration explained before, the polygon corresponding to the
filtration $0 \subset  \fF' \subset f^\ast\fE $ lies under the Harder-Narasimhan polygon of $f^\ast\fE $.
Since $\mu_{min}(f^\ast\fE ) \geq0$, this immediately implies that $\deg\fF' \leq \deg  f^\ast\fE$ and hence
$\deg \fF\geq 0$, so that $\mathfrak L_S$ is nef.  This shows that $\fE$ is H-nef, thereby completing the
proof of the lemma.
\end{proof}

\begin{lemma}\label{H-nflat}
  If $f\,\colon\, Y \,\longrightarrow\, X$ is a  surjective morphism of smooth projective
    varieties, and $\fE$ is a Higgs bundle on $X$, then $\fE$ is  H-nef if and only if $f^\ast\fE$ is. 
\end{lemma}
\begin{proof}
 Suppose $\fE$ is H-nef. Consider $f^*(\fE)$. Let $\psi: C \rightarrow Y $ be any morphism from a smooth projective curve $C$ to $Y$. Then 
 deg$(\psi^*f^*(\fE))= \text{deg}((f\circ\psi)^*(\fE)) \geq0$ by H-nefness of $\fE$. Hence det$(f^*(\fE))$ is nef. Also $\mu_{min}(\psi^*f^*(\fE))=\mu_{min} ((f\circ\psi)^*(\fE)) \geq0$ by H-nefness of $\fE$ thereby showing that $f^*(\fE)$ is H-nef.
 Conversely suppose $f^*(\fE)$ is H-nef on $Y$. Let $\phi: C \rightarrow X$ be any morphism from a smooth projective curve to $X$.
 Then there exists a surjective morphism from a smooth projective curve $g: \tilde{C} \rightarrow C$ and a morphism $\tilde{\phi}: \tilde{C} \rightarrow Y$ lying 
 over the morphism $\phi$. Then deg$(f \circ \tilde{\phi})^*(\fE)\geq0$ by H-nefness of $f^*(\fE)$ and hence by the commutativity of the diagram, deg$(\phi \circ g)^*{\fE} \geq0$. Since $g$ is a finite morphism this shows that deg$(\phi^*(\fE))\geq0$ thus proving that $det(\fE)$ is nef.
Similarly $\mu_{min}(f \circ \tilde{\phi})^*(\fE) \geq 0 $ by H-nefness of $f^*(\fE)$ and hence by the commutativity of the diagram,
 $\mu_{min}(\phi \circ g)^*(\fE) \geq 0$. Since $g$ is a finite morphism, it follows that $\mu_{min}(\phi^*(\fE))\geq 0$ as well.
\end{proof}
     
\begin{lemma}
  Every quotient Higgs bundle of an H-nef Higgs bundle
    $\fE$ on $X$ is H-nef.
\end{lemma}

\begin{proof}
  Let $\fE \twoheadrightarrow \fE''$ be a non-trivial Higgs quotient. Let $\fE'$ denote the kernel. 
 Let $f:C \rightarrow X$ be a morphism from a smooth curve $C$ to $X$. By the property of the (Higgs) Harder-Narasimhan filtration
 mentioned earlier, $\mu_{min}(f^*(\fE))\geq0$ and hence deg$(f^*(\fE'))\leq$ deg$(f^*(\fE))$. Thus deg$(f^*(\fE''))\geq0$ proving that det$(\fE'')$ is nef. 
 To prove the second condition, let $f^*(\fE'') \rightarrow \fF$
 be a Higgs quotient. Then $\fF$ is also a Higgs quotient of $f^*(\fE)$ and hence by H-nefness of $\fE$, deg$(\fF)\geq 0$. This shows that $\mu_{min}(f^*(\fE''))\geq 0$
thus completing the proof that $\fE''$ is H-nef as well. 
\end{proof}

The remaining results in this section will be the key to prove that H-nflat Higgs bundles make up a Tannakian category. 
\begin{thm}
Let $X$ be a smooth projective variety. Let $\fE$ and $\fF$ be two H-nef bundles on $X$. Then $\fE\otimes \fF$
is also H-nef.
\end{thm}

\begin{proof}
Let $f\colon C \rightarrow X$ be any smooth projective curve mapping to $X$. Since $\fE$ and $\fF$ are both
H-nef,
$\mu_{\text{min}}(f^\ast\fE)$ and $\mu_{\text{min}}(f^\ast\fF)$ are both non-negative. By property 2 of the
Harder-Narasimhan-filtration explained earlier,  $\mu_{\text{min}}(f^\ast(\fE\otimes \fF))\geq 0$.
Hence by Lemma \ref{le1}, the tensor product $\fE\otimes \fF$ is also H-nef.
\end{proof}

Finally, we have the following property of morphisms between H-nflat Higgs bundles.

\begin{prop}\label{kerncokern}
Let $\beta \,\colon\, \fE\,=\, (E,\,\phi)\,\longrightarrow\,\fF\,=\,(F,\,\psi)$ be a morphism of H-nflat Higgs bundles
on a smooth projective variety $X$. The kernel and cokernel of $\beta$ are both locally free. 
\end{prop}

\begin{proof}
The proposition is equivalent to the statement that
$\dim \beta(E_x)$ is independent of $x\, \in\, X$. Therefore, it suffices to show
the following: for every pair $(C\, ,f)$, where
$C$ is a smooth projective curve and $f: C\rightarrow X$ is a morphism,
the image $(f^*\beta) (f^*E)$ is a subbundle of $f^*F$.

{}From Lemma \ref{le1} we know that $\fE$ and $\fF$ are Higgs semistable of degree zero.
Therefore, it is enough to prove the proposition for smooth projective curves.

So take $X$ to be a smooth projective curve. Take semistable Higgs
bundles $\fE= (E,\phi)$ and $\fF=(F,\psi)$ of degree zero on $X$, and let
$\beta \,\colon\, \fE\,\longrightarrow\,\fF$ be a nonzero homomorphism. 
Since $\beta(E)$ is a quotient of $E$ (respectively, subsheaf of $F$), we have
$\deg \beta(E)\geq 0$ (respectively, $\deg \beta(E)\leq 0$). Therefore, it follows that
\begin{equation}\label{g2}
\deg \beta(E) \,=\, 0\, .
\end{equation}

Next, we will show that the quotient $F/\beta(E)$ is torsion-free. Let $T$
be the torsion part of $F/\beta(E)$. Let $F'$ be the inverse image of $T$ in $F$.
We have
$$
\deg F' \,= \, \deg \beta(E) + \deg T \,=\, \deg T\, .
$$
So if $T\,\not=\, 0$, then $\deg F'\,=\, \deg T\, > \,0$, and hence in this case
$F'$ contradict the semistability condition for $\fF$. Consequently, we have
$T \,=\, 0$. This implies that $\beta(E)$ is a subbundle of $F$.
\end{proof}

\begin{prop}\label{kerncokern2}
Let $\beta \,\colon\, \fE\,=\, (E,\,\phi)\,\longrightarrow\,\fF\,=\,(F,\,\psi)$ be a morphism of H-nflat Higgs bundles
on a smooth projective variety $X$. The kernel and cokernel of $\beta$ are
H-nflat Higgs bundles.
\end{prop}

\begin{proof}
{}From Proposition \ref{kerncokern} we know that both kernel and cokernel of $\beta$
are locally free. As in the proof of Proposition \ref{kerncokern}, take $X$ to be
a smooth projective curve. Then by proposition \ref{results}, $\fE$ and $\fF$ are semistable of degree $0$. 
Since $\deg E\,=\,0$, from eq.\,\eqref{g2} it follows immediately
that $\deg (\ker~\beta)\,=\, 0$ if $\beta\,\not=\, 0$. Similarly, since
$\deg F\,=\,0$, from eq.\,\eqref{g2} it follows immediately
that $\deg (\coker~\beta)\,=\, 0$ if $\beta\,\not=\, 0$. Since $E$ and $F$   are Higgs-semistable of degree zero,
and $\ker~\beta$ and $\coker~\beta$ are of degree $0$, it follows that $\ker~\beta$ and $\coker~\beta$ are also Higgs-semistable of degree zero. Since the pullbacks of  $\ker~\beta$ and $\coker~\beta$ to  any smooth curve are Higgs semistable of 
degree $0$, by Proposition \ref{results}, it follows that 
both kernel and cokernel of $\beta$ are H-nflat Higgs bundles.
\end{proof}

\bigskip
\section{Categories of numerically flat bundles}\label{last}
  
\begin{defin} Given a smooth projective variety $X$ over a field $k$ of characteristic zero, we consider the
following categories.
\begin{enumerate} \item The category $\mathbf{NF}(X)$ whose objects are numerically flat vector bundles on
$X$, and morphisms are morphisms of vector bundles (i.e., kernel and cokernel are local free).
\item The category $\mathbf{HNF}(X)$ whose objects are H-numerically flat Higgs bundles on $X$, and
morphisms are morphisms of Higgs bundles (i.e., kernel and cokernel are locally free, and the kernel is
invariant under the Higgs field.).
\end{enumerate}
\end{defin}
$\mathbf{NF}(X)$ and $\mathbf{HNF}(X)$ are Abelian categories (the case of $\mathbf{HNF}(X)$ follows as a consequence of  Proposition \ref{kerncokern}), and $\mathbf{NF}(X)$ is a proper subcategory
of $\mathbf{HNF}(X)$. Both are tensor categories (cf.\ in particular Proposition \ref{results} (iv)). Moreover,
they are rigid in the sense of \cite{deligne-milne}, Definition 1.7.

We remind the reader that a {\em neutral Tannakian category over a field $k$} is a rigid Abelian $k$-linear
tensor category $\mathbf C$
together with a faithful $k$-linear tensor functor $\omega\colon\mathbf C\,\longrightarrow\, \mathbf{Vect}_k$. Here
$\mathbf{Vect}_k$ is the category of $k$-vector spaces, and $\omega$ is called the {\em fiber functor.} Then,
there exists an affine group scheme $G$ over $k$ such that $\mathbf C$
is equivalent to the category $\mathbf{Rep}_k(G)$ of $k$-linear representations of $G$ (see
\cite{deligne-milne}).

 $\mathbf{HNF}(X)$ is indeed a neutral Tannakian category (with $\omega$ the functor that associates to an
H-flat Higgs bundle $\fE=(E,\phi)$ its fiber $E_x$ at a fixed point $x\in X$), so that the following
definition makes sense.

\begin{defin} Let $x\in X$. The {\em Higgs fundamental group scheme} $\pi_1^H(X,x)$ is the affine group scheme
representing 
the category  $\mathbf{HNF}(X)$ with the fiber functor $\fE \,\longmapsto\, E_x$.
\end{defin}

If $\pi_1^S(X,x)$ is the   fundamental group scheme associated with the category $\mathbf{NF}(X)$  \cite{langer-fund-I}, the
inclusion
$\mathbf{NF}(X)\hookrightarrow \mathbf{HNF}(X)$ induces a  faithfully flat homomorphism of group schemes $\pi_1^H(X,x)
\,\longrightarrow\,\pi_1^S(X,x)$.

We conclude this paper by giving a few properties of the Higgs fundamental group scheme. A more thorough
study of this group will form the object of a future paper. 

\begin{prop} Let $f\,\colon\, X'\,\longrightarrow\, X$ be a surjective, flat morphism of projective varieties over $k$. If
$f_\ast\cO_{X'}\,\simeq\,\cO_X$ and $f(x')\,=\,x$, then
the induced morphism $\pi_1^H(X',\,x') \,\longrightarrow\, \pi_1^H(X,\,x)$ is a surjective faithfully flat morphism.
\end{prop}

\begin{proof}
 By \cite[Prop.~2.21(a)]{deligne-milne}, it suffices to show that if $E$ is an H-numerically flat bundle on 
$X$ and $F'\subset f^* E\,=\,E'$ (say) is an H-numerically flat subbundle of $E'$ on $X'$, then 
there exists an H-numerically flat subbundle $F \,\subset \,E$ on $X$ such that $f^* F\,=\, F'$. 
Fix $y \in X$. Let $E_y$ (respectively, $E'_y,F'_y$) denote the restrictions of $E$ (respectively, $E'$, 
$F'$) to $y$ (respectively, $X'_y$). Consider the surjection $E_y^{\prime*} \twoheadrightarrow 
F_y^{\prime*}$ corresponding to the inclusion $F'_y \subseteq E'_y$. Since $E'_y$ is trivial and 
hence globally generated, it follows that $F_y^{\prime*}$ is globally generated as well. But since 
$c_1(F'_y)$ is numerically equivalent to zero, it follows that any section of $F_y^{\prime*}$ has 
no zero's and hence $F'^{*}$ and therefore $F'$ is trivial on the fibers of $f$. Since by flatness of $f$,
$h^0(F'|_{X'_y})$ is independent of $y\in X$, by Grauert's theorem it follows that $f_* F'$ is 
locally free. This and the given condition that $f_\ast\cO_{X'}\,\simeq\,\cO_X$ together
imply that the natural map $f^*f_*F' \,\longrightarrow\, F'$ is an isomorphism 
of bundles. Taking $F$ to be $f_*F'$ thereby produces a subbundle $F\subseteq E$ such that $f^*F$ is 
isomorphic to $F'$. It is easy to see that $F$ is invariant under the Higgs field on $E$.
The vector bundle $F$ equipped with the induced Higgs field is also H-nflat by Proposition \ref{H-nflat} since its pullback under $f$ is H-nflat, thereby completing the 
proof of the proposition.
\end{proof}

We also mention the following facts.

\begin{itemize}\item If $\pi^H_1(X,x) = \{e\}$, the category $\mathbf{HNF}(X)$ is equivalent
to the category $\mathbf{Vect_k}$ of finite-dimensional vector spaces. As a consequence,
all H-nflat Higgs bundles are trivial.
\item If the natural morphism $\pi^H_1(X,x) \to \pi^S_1(X,x) $ is an isomorphism,
the categories $\mathbf{HNF}(X)$ and $\mathbf{NF}(X)$ are equivalent. This means that
all H-nflat Higgs bundles   only have zero Higgs field, which also
implies that the Conjecture \ref{conj} holds true.  \end{itemize}

In \cite{bruzzo-logiudice} a characterization was given of some classes of varieties for which the conjecture holds (basically, varieties with nef tangent bundle).

Let $X$, $Y$ be projective varieties over $k$, and let $x$, $y$ be points in $X$, $Y$, respectively.
Given Higgs bundles $(E,\, \theta)$ and $(F,\, \phi)$ on $X$ and $Y$ respectively, we have
the Higgs bundle $(E\boxtimes F,\, \theta\otimes{\rm Id}+ \text{\rm Id}\otimes\phi)$ on $X\times Y$.
This construction produces a homomorphism
\begin{equation}\label{product} \pi_1^H(X\times_k Y,\, (x,y)) \,\longrightarrow\, \pi_1^H(X,x)\times\pi_1^H(Y,y)\, . 
\end{equation}
At the moment we do not know whether
the above homomorphism is an isomorphism. This fact, via   Corollary \ref{equiv}, is related
to the conjecture that Theorem \ref{thmvb} also holds for Higgs bundles. 
 If   indeed  the morphism \eqref{product} is an isomorphism, then
any numerically flat Higgs bundle on $C_1\times\ldots \times C_d$, where $C_i$ are smooth
projective curves, would arise from numerically flat Higgs bundles on the curves $C_i$. A
numerically flat Higgs bundles on a curve is of degree zero. Therefore, all higher Chern classes
of a numerically flat Higgs bundle on $C_1\times\ldots \times C_d$ would be numerically equivalent
to zero. On the other hand, the numerical vanishing of higher Chern classes of
a numerically flat Higgs bundle is the key obstruction if one tries the generalize
of proof of the product formula for the usual numerically flat case (no Higgs field), as given in
\cite{La2}, to   Higgs bundles.

\end{document}